\documentclass[a4paper,12pt,dvipdfmx]{amsart}
\usepackage{amsmath}
\usepackage{amssymb}
\usepackage{enumerate}
\usepackage{amscd}
\usepackage{graphics}
\usepackage{latexsym}
\usepackage{verbatim} 
\usepackage{url}

\theoremstyle{plain}
\newtheorem{thm}{{\bf Theorem}}[section]
\newtheorem{cor}[thm]{{\bf  Corollary}}
\newtheorem{prop}[thm]{{\bf Proposition}}
\newtheorem{lemma}[thm]{{\bf Lemma}}

\theoremstyle{definition}
\newtheorem{define}[thm]{{\bf Definition}}
\newtheorem{note}[thm]{{\bf Note}}

\newtheorem{question}[thm]{{\bf Question}}

\newcommand{\col}{\mathord{\mathrm{Col}}}

\newcommand{\seq}[1]{\langle {#1} \rangle}

\newcommand{\norm}[1]{\left\| {#1} \right\|}

\newcommand{\ka}{\kappa}
\newcommand{\la}{\lambda}
\newcommand{\om}{\omega}

\newcommand{\bbP}{\mathbb{P}}
\newcommand{\bbQ}{\mathbb{Q}}

\newcommand{\bbM}{\mathbb{M}}

\newcommand{\calF}{\mathcal{F}}
\newcommand{\calG}{\mathcal{G}}

\newcommand{\HOD}{\mathrm{HOD}}

\newcommand{\ZF}{\mathsf{ZF}}
\newcommand{\ZFC}{\mathsf{ZFC}}

\newcommand{\AC}{\mathsf{AC}}
\newcommand{\SVC}{\mathsf{SVC}}
\newcommand{\ON}{\mathrm{ON}}
\newcommand{\LS}{L\"owenheim-Skolem }
\newcommand{\NBG}{\mathsf{NBG}}
\newcommand{\CLS}{\mathsf{CLS}}

\title[Geology of symmetric grounds]{Geology of symmetric grounds}
\author[T. Usuba]{Toshimichi Usuba}
\address[T. Usuba]
{Faculty of Fundamental Science and Engineering,
Waseda University, 
Okubo 3-4-1, Shinjyuku, Tokyo, 169-8555 Japan}
\email{usuba@waseda.jp}

\begin{document}
\begin{abstract}
Let us say that a model of $\ZF$ is a symmetric ground
if $V$ is a symmetric extension of the model.
In this paper, we investigate set-theoretic geology of symmetric grounds.
Under a certain assumption,
we show that all symmetric grounds of $V$ are uniformly definable.
We also show that if $\AC$ is forceable over $V$,
then the symmetric grounds are downward directed.
\end{abstract}

\maketitle

\section{Introduction}
Let us say that a transitive model of $\ZF$ is a \emph{$\ZF$-ground}, or simply a \emph{ground}, of $V$
if $V$ is a generic extension of the model.
If a ground satisfies $\AC$, the Axiom of Choice,
then it is a \emph{$\ZFC$-ground}.
In $\ZFC$, Fuchs-Hamkins-Reitz \cite{FHR} studied the structure of all $\ZFC$-grounds of $V$,
it is called \emph{set-theoretic geology}.
Their work was under $\AC$, and Usuba \cite{U2} tried to study \emph{set-theoretic geology without $\AC$};
The universe $V$ and each ground are not assumed to satisfy $\AC$.
These are attempts to investigate the nature of forcing method.

When we want to build choiceless models,
the method of \emph{symmetric submodel}, or \emph{symmetric extension}, is a very powerful and flexible tool.
For a generic extension $V[G]$ of $V$, a symmetric submodel of $V[G]$, or a symmetric extension of $V$, 
is realized as a submodel of the generic extension $V[G]$.
The model $V[G]$ itself is a symmetric submodel of $V[G]$, so every generic extension is a symmetric extension.
Moreover symmetric extensions have many properties which are parallel to generic extensions,
such as forcing relation and forcing theorem.

We want to say that a model $W$ of $\ZF$ is a \emph{symmetric ground} of $V$ if
$V$ is a symmetric extension of $W$, that is, $V$ is a symmetric submodel of \emph{some} generic extension of $W$.
Now we can expect to extend standard set-theoretic geology to one which treats 
symmetric extensions and symmetric grounds.
If $V$ is a symmetric extension of some model,
then $V$ could be a choiceless model.
So our base theory should be $\ZF$. 
On the other hand, our definition of symmetric grounds causes a problem:
What is \emph{some} generic extension of $W$?
It would be possible that there is a $W$-generic $G$ which is 
\emph{not} living in $V$ and in any generic extension of $V$, but $V$ is  realized as a 
symmetric submodel of $W[G]$.
Hence, in $V$, it is not clear if we can describe that ``$W$ is a symmetric ground of $V$''
and develop geology of symmetric grounds.

For this problem, using Grigorieff's work \cite{G},
we prove that the statement ``$W$ is a symmetric ground of $V$'' is actually describable
in $V$ by a certain first order formula of the extended language $\{\in, W\}$.
We also prove that, under a certain assumption,
all symmetric grounds are uniformly definable by a first order formula of set-theory.
Let $\CLS$ denote the assertion that there is a proper class of \LS cardinals (see Section 3).
\begin{thm}[in $\ZF$]
Suppose $\CLS$.
Then there is a first order formula $\varphi(x,y)$ of set-theory such that:
\begin{enumerate}
\item For every set $r$, $\overline{W}_r=\{x \mid \varphi(x,r)\}$ is a symmetric ground of $V$ with $r \in \overline{W}_r$.
\item For every symmetric ground $W$ of $V$, there is $r$ with
$W=\overline{W}_r$.
\end{enumerate}
\end{thm}
$\CLS$ follows from $\AC$.
Hence in $\ZFC$, all symmetric grounds are uniformly definable.
We also show that if $V$ satisfies $\AC$,
then $V$ is definable in any symmetric extension of $V$.

The above uniform definability of symmetric grounds allows us to investigate  \emph{set-theoretic geology of symmetric grounds} in $\ZF$, which is a study of the structure of all symmetric grounds.
This paper is a first step in  the set-theoretic geology of symmetric grounds.
In $\ZF$, it is consistent that there are two grounds which have no common ground (\cite{U2}).
We show that, if $\AC$ is forceable over $V$ then
all symmetric grounds are downward directed.
Moreover the intersection of all symmetric grounds is a model of $\ZFC$ if $\AC$ is forceable over $V$.

\section{Preliminaries}
\textbf{Throughout this paper, we do not assume $\AC$ unless otherwise specified.}
Forcing means a set forcing.
A \emph{class} 
means a second order object in the sense of
Von Neumann-Bernays-G\"odel set-theory $\mathsf{NBG}$
unless otherwise specified.
We do not require that a class $M$ is definable in $V$ with some
parameters, but
we assume that $V$ satisfies the comprehension and replacement scheme 
for the formulas of the language $\{\in, M\}$
(where we identify $M$ as a unary predicate).
Note that, if $M$ is a definable class by a formula of the language $\{\in\}$,
then $V$ satisfies the comprehension and replacement scheme 
for the formulas of the language $\{\in, M\}$.
Hence every definable class is a \emph{class} in our sense.
We also note that $V$ is a class of any generic extension of $V$
by the forcing theorem.
Any theorems in this paper involving classes can be formalized 
in $\NBG$, or some small extension of $\ZF$.\footnote{
For instance,
we extend the language of set-theory by adding unary predicates $M, N, W,\dotsc$
and let $T$ be the theory $\ZF$ together with the comprehension and replacement scheme
for the formulas of the language $\{\in, M, N,W,\dotsc\}$.}

To treat classes in generic extensions, we extend the forcing language and 
relation as follows. 
Let $M$ be a class.
First, we extend the forcing language
by adding the symbol $\check M$.
For a poset $\bbP$, $p \in \bbP$, and a $\bbP$-name $\dot x$,
we define $p \Vdash_\bbP \dot x \in \check M$ if
the set $\{q \le p \mid$ there is $x \in M$ with $q \Vdash_\bbP x =\dot x \}$
is dense below $p$.
$\check M$ can be treated as a class-sized canonical name, that is,
$\check M=\{\seq{\check x, p} \mid p \in \bbP, x \in M\}$,
where $\check x$ is the canonical name for $x$.
By the standard way, we define $p \Vdash_\bbP \varphi$ for every formula $\varphi$
of the extended forcing language.
Since $M$ is a class of $V$, we can easily check that the forcing theorem holds for the 
extended forcing language and relation,
and $M=\{\dot x^G \mid p \Vdash_\bbP \dot x \in \check M$ for some $p \in G\}$.
Using the extended forcing theorem,
we have the following:
If $M$ is a class of $V$, then it is a class of any generic extension of $V$.

A model of $\ZF(\mathsf{C})$ means a transitive class model of $\ZF(\mathsf{C})$ containing all ordinals.
It is known that a transitive class $M$ containing all ordinals
is a model of $\ZF$ if and only if
$M$ is closed under the G\"odel operations and almost universal,
that is, for every set $X \subseteq M$ there is a set $Y \in M$ with
$X \subseteq Y$.
So we can identify the sentence ``$M$ is a model of $\ZF$''
with the conjunction of the following sentences of the language $\{\in, M\}$:
\begin{enumerate}
\item $M$ is transitive and contains all ordinals.
\item $M$ is closed under the G\"odel operations.
\item $M$ is almost universal.
\end{enumerate}
A model $W$ of $\ZF$ is a \emph{$\ZF$-ground}, or simply \emph{ground}, of $V$ if
there are a poset $\bbP \in W$ and
a $(W, \bbP)$-generic $G$ with $V=W[G]$.
If a $\ZF$-ground satisfies $\AC$, then it is a \emph{$\ZFC$-ground}.
\textbf{\textbf From now on,  a ground means a $\ZF$-ground unless otherwise specified.}

For a model $M$ of $\ZF$ and an ordinal $\alpha$, let $M_\alpha=M \cap V_\alpha$,
the set of all $x \in M$ with rank  less than $\alpha$.

Here we present a series of results by Grigorieff \cite{G},
which we will use frequently.

\begin{thm}[\cite{G}]
For models $M$, $N$ of $\ZF$,
if $M \subseteq N$, $M$ is a ground of $N$, and $N$ is a ground of $V$,
then $M$ is a ground of $V$.
\end{thm}
This means that the ``is a ground of'' relation is transitive in $\ZF$.

For a model $M$ of $\ZF$ and a set $X$,
let $M(X)=\bigcup_{\alpha \in \mathrm{ON}} L(M_\alpha \cup \{X\})$.
$M(X)$ is the minimal model of $\ZF$ with $M\subseteq M(X)$ and $X \in M(X)$.
Note that $M(X)$ is also a class of $V$.

\begin{thm}[Theorem B in \cite{G}]\label{1.2}
Let $M \subseteq V$ be a ground of $V$.
Let $N$ be a model of $\ZF$ 
and suppose $M \subseteq N \subseteq V$.
Then the following are equivalent:
\begin{enumerate}
\item $V$ is a generic extension of $N$.
\item $N$ is of the form $M(X)$ for some $X \in N$.
\end{enumerate}
\end{thm}

For a  class $X$,
let $\HOD_X$ be the collection of all hereditarily definable sets
with parameters from $\ON \cup X$.
If $M$ is a model of $\ZF$,
it is known that 
$\HOD_M$ is also a model of $\ZF$ with $M \subseteq \HOD_M$
\footnote{Note that $\HOD_M$ is not the relativisation of $\HOD$ to $M$.},
and $\HOD_M$ is a class of $V$ (e.g., see \cite{G}).
We note that $M$ is a class of $\HOD_M$ if $M$ is a model of $\ZF$.

\begin{thm}[9.3 Theorem 1 in \cite{G}]\label{1.3}
Let $M$ be  a model of $\ZF$ such that
$V=M(X)$ for some $X \in V$.
Then $V$ is a ground of some generic extension of $\HOD_M$,
and $\HOD_M=M(Y)$ for some $Y \in V$ with $Y \subseteq M$.
\end{thm}

\begin{note}\label{1.4}
One may wonder what is ``some generic extension of $\HOD_M$''.
This theorem can be justified as follows:
There is a poset $\bbP$
such that, in $V$, $\bbP$ forces that ``$(\HOD_M)^{\check{}}$ is a ground of the universe''.
\end{note}

For a set $S$,
let $\col(S)$ be the poset consists of all finite partial functions from
$\om$ to $S$ ordered by  reverse inclusion.
$\col(S)$ is weakly homogeneous, and if $S$ is ordinal definable then so is $\col(S)$.
\begin{thm}[4.9 Theorem 1 in \cite{G}]\label{1.5}
Let $\bbP$  be a poset, and $G$ be $(V, \bbP)$-generic.
Let $\alpha$ be a limit ordinal with $\alpha>\mathrm{rank}(\bbP)\cdot \om$.
Let $H$ be $(V[G], \col(V[G]_\alpha))$-generic.
Then there is a $(V, \col(V_\alpha))$-generic $H'$ with
$V[G][H]=V[H']$.
\end{thm}
The following is just rephrasing of this theorem:
\begin{lemma}\label{1.5+}
Let $M,N$ be  grounds of $V$, and $\alpha$ a sufficiently large limit ordinal.
Then  whenever $G$ is $(V, \col(V_\alpha))$-generic,
there are an $(M, \col(M_\alpha))$-generic $H$ and an $(N, \col(N_\alpha))$-generic $H'$ with
$V[G]=M[H]=N[H']$.
\end{lemma}


We recall some definitions and facts about symmetric submodels and extensions.
See Jech \cite{J} for more details.
Let $\bbP$ be a poset, and $\mathrm{Auto}(\bbP)$ the group of the automorphisms on $\bbP$.
For $\pi \in \mathrm{Auto}(\bbP)$, 
$\pi$ can be extended to the (class) map from the $\bbP$-names to $\bbP$-names canonically.
Let $\calG$ be a subgroup of $\mathrm{Auto}(\bbP)$.
A non-empty family $\calF$ of subgroups of $\calG$ is a \emph{normal filter} on $\calG$ if:
\begin{enumerate}
\item $H, H' \in \calF \Rightarrow H \cap H' \in \calF$.
\item $H \in \calF$, $H'$ is a subgroup of $\calG$ with $H \subseteq H'$
then $H'\in \calF$.
\item For $H \in \calF$ and $\pi \in \calG$, we have 
$\pi^{-1} H \pi \in \calF$.
\end{enumerate}
A triple $\seq{\bbP, \calG, \calF}$ is a \emph{symmetric system}
if $\calG$ is a subgroup of $\mathrm{Auto}(\bbP)$ and
$\calF$ is a normal filter on $\calG$.
A $\bbP$-name $\dot x$ is
\emph{$\calF$-symmetric} if $\{ \pi \in \calG \mid \pi(\dot x)=\dot x\} \in \calF$.
Let $\mathrm{HS}_{\calF}$ be the class of all hereditarily $\calF$-symmetric names.

If $G$ is $(V, \bbP)$-generic,
then $\mathrm{HS}_{\calF}^G=\{\dot x^G \mid \dot x \in \mathrm{HS}_\calF\}$
is a transitive model of $\ZF$ with $V \subseteq\mathrm{HS}_{\calF}^G$.
$\mathrm{HS}_{\calF}^G$ is called 
a \emph{symmetric submodel of $V[G]$}, or a \emph{symmetric extension of $V$}.
If we want to specify a poset $\bbP$,
we say that $\mathrm{HS}_{\calF}^G$ is a symmetric extension of $V$ via a poset $\bbP$,
or a symmetric submodel of $V[G]$ via a poset $\bbP$.
We note that 
$\mathrm{HS}_{\calF}^G$ is a class of $V[G]$, and $V$ is a class of $\mathrm{HS}_{\calF}^G$.

\begin{note}\label{2.7-}
Let $M$ be a model of of $\ZF$,
$V[G]$ a generic extension of $V$ via a poset $\bbP$,
and $V \subseteq M \subseteq V[G]$.
One may expect that if $M$ is a symmetric extension of $V$,
then $M$ is a symmetric submodel of $V[G]$
via a poset $\bbP$. However this is not correct.
This follows from the construction of the Bristol model (Karagila \cite{K}).
The Bristol model $M$ is an intermediate model between
$L$ and the Cohen forcing extension $L[c]$,
and $M$ cannot be of the form $L(X)$ for some set $X$.
$M$ is of the form $\bigcup_{\alpha \in \mathrm{ON}} L(M_{\alpha})$,
and each $L(M_{\alpha})$ is a symmetric extension of $L$,
but there is some $\beta$ such that for every  $\alpha \ge \beta$,
$L(M_\alpha)$ is not a symmetric submodel of $L[c]$ via the Cohen forcing notion.
\end{note}
The following theorem tells us when $M \subseteq V[G]$ is a symmetric submodel of $V[G]$.

\begin{thm}[Theorem C in \cite{G}]\label{1.6}
Let $M \subseteq V$ be a model of $\ZF$, $\bbP \in M$ a poset, and
$G$ be $(V, \bbP)$-generic and suppose $V=M[G]$.
Let $N$ be a model of $\ZF$ with $M \subseteq N \subseteq V$.
Then $N$ is a symmetric submodel of $V$ via $\mathrm{ro}(\bbP)^M$ if and only if
$N$ is of the form $\mathrm{HOD}_{M(X)}$ for some
$X \in N$, where $\mathrm{ro}(\bbP)$ is the completion of $\bbP$.\footnote{We can take the completion of $\bbP$ without $\AC$,
see \cite{G}.}
\end{thm}

\section{Choiceless geology}
In this section, we recall some definitions and facts about set-theoretic geology without $\AC$.
See Usuba \cite{U2} for more information.
For \LS cardinals below, see also \cite{U3}.

We use the notion of \emph{\LS cardinal},
which was introduced in \cite{U2}.
The notion of \LS cardinal corresponds to the downward 
\LS theorem in the context of $\ZFC$.
The downward \LS theorem is not provable from $\ZF$,
but under the existence of \LS cardinals,
we can take  \emph{small} elementary submodels of $V_\alpha$ for every $\alpha$.

\begin{define}
An uncountable cardinal $\ka$ is a \emph{\LS cardinal}
if for every $\alpha \ge \ka$, $x \in V_\alpha$, and $\gamma<\ka$,
there are $\beta \ge \alpha$ and $X \prec V_\beta$ satisfying the following:
\begin{enumerate}
\item $x \in X$ and $V_\gamma \subseteq X$.
\item ${}^{V_\gamma} (X \cap V_\alpha) \subseteq X$.
\item The transitive collapse of $X$ belongs to $V_\ka$.
\end{enumerate}
\end{define}
Note that every limit of \LS cardinals
is a \LS cardinal as well.

\begin{thm}[\cite{U2}]\label{2.2}
Suppose $\ka$ is a limit of \LS cardinals.
Then for every poset $\bbP \in V_\ka$, $\bbP$ forces that
``$\ka$ is a \LS cardinal''.
Conversely, if $\ka<\la$ are cardinals,
$\bbP \in V_\ka$ a poset, and $\Vdash_\bbP$``$\ka,\la$ are \LS'',
then $\la$ is \LS in $V$.
\end{thm}

\begin{define}
For a set $X$, the \emph{norm} of $X$, denoted by $\norm{X}$, is the least ordinal $\alpha$
such that there is a surjection from $V_\alpha$ onto $X$.
\end{define}

\begin{define}
Let $M \subseteq V$ be a model of (a sufficiently large fragment of ) $\ZF$,
and $\alpha$ an ordinal.
\begin{enumerate}
\item $M$ has the \emph{$\alpha$-norm covering property} for $V$
if for every set $X \subseteq M$, if $\norm{X}<\alpha$ then 
there is $Y \in M$ such that $X \subseteq Y$ and $\norm{Y}^M<\alpha$.
\item $M$ has the \emph{$\alpha$-norm approximation property} for $V$
if for every set $Y \subseteq M$, if $Y \cap a \in M$
for every $a \in M$ with $\norm{a}^M<\alpha$
then $Y \in M$.
\end{enumerate}
\end{define}

\begin{thm}[\cite{U2}]\label{2.5}
Suppose $\ka$ is a \LS cardinal and $\alpha<\ka$.
For all models $M, N$ of (a sufficiently large fragment of) $\ZF$,
if $M$ and $N$ have the $\alpha$-norm covering and the $\alpha$-norm approximation properties for $V$
and $M_\ka=N_\ka$,
then $M=N$.
\end{thm}

\begin{thm}[\cite{U2}]\label{2.6}
Let $\ka$ be a  \LS  cardinal, and $W \subseteq V$ a model of $\ZF$.
If there are a poset $\bbP \in W_\ka$ and 
a $(W, \bbP)$-generic $G$ with $V=W[G]$,
then $W$ has the  $\ka$-norm covering and the $\ka$-norm approximation properties for $V$.
\end{thm}
Let us say that a model $M$ of $\ZF$ is a \emph{pseudo-ground}
if $M$ has  the $\alpha$-norm covering and the $\alpha$-norm approximation properties for $V$
for some $\alpha$.

The following is immediate:
\begin{lemma}\label{2.7}
Let $M \subseteq N \subseteq V$ be models of $\ZF$, and $\alpha$ an ordinal.
\begin{enumerate}
\item If $M$ has the $\alpha$-norm covering and the $\alpha$-norm approximation properties for $N$,
and $N$  for $V$,
then $M$ has the $\alpha$-norm covering and the $\alpha$-norm approximation properties for $V$.
\item If $M$ has the $\alpha$-norm covering and the $\alpha$-norm approximation properties for $V$,
then $M$ has the $\alpha$-norm covering and the $\alpha$-norm approximation properties for $N$.
\end{enumerate}
\end{lemma}

\begin{define}
Let $\CLS$ denote the assertion
that ``there is a  proper class of \LS cardinals''.
\end{define}
By Theorem \ref{2.2},
$\CLS$ is absolute between $V$, all generic extensions, and all grounds of $V$.

\begin{thm}[\cite{U2}]\label{2.9}
Suppose $\CLS$.
Then all pseudo-grounds are uniformly definable:
There is a first order formula $\varphi'(x,y)$ of set-theory such that:
\begin{enumerate}
\item For every $r \in V$, the class $W_r'=\{x \mid \varphi'(x,r) \}$ is a pseudo-ground
of $V$ with $r \in W_r'$.
\item For every pseudo-ground $W$ of $\ZF$,
there is $r$ with $W=W_r'$.
\end{enumerate}
\end{thm}
We sketch the proof since we will need to know how to define $\varphi'$ and $W'_r$ later.
\begin{proof}[Sketch of the proof]
For $r \in V$, suppose $r$ fulfills the following properties:
\begin{enumerate}
\item $r$ is of the form $\seq{X, \ka, \alpha}$ where $\ka$ is a \LS cardinal,
$\alpha< \ka$, and $X$ is a transitive set with $X \cap \mathrm{ON}=\ka$.
\item For each cardinal $\la>\ka$, if $V_\la$ is a model of a sufficiently large
fragment of $\ZF$, then there is a unique transitive model $X^{r,\la}$ of
a sufficiently large fragment of $\ZF$ such that
$X^{r,\la} \cap \mathrm{ON}=\la$, $(X^{r,\la})_\ka=X$, and 
$X^{r,\la}$ has the $\alpha$-norm covering and the $\alpha$-norm approximation properties for $V_\la$.
\end{enumerate}
Then let $W_r'=\bigcup_{\la>\ka} W^{r,\la}$.
Otherwise, let $W_r'=V$.
By Theorem  \ref{2.5},
one can check that the collection $\{W_r' \mid r\in V\}$ is as required.
\end{proof}
Using this, we have the uniform definability of all grounds under $\CLS$.
\begin{thm}[\cite{U2}]\label{unif. def.  grounds}
Suppose $\CLS$.
Then there is a first order formula $\varphi(x,y)$ of set-theory such that:
\begin{enumerate}
\item For every $r \in V$, the class $W_r=\{x \mid \varphi(x,r) \}$ is a ground
of $V$ with $r \in W_r$.
\item For every ground $W$ of $\ZF$,
there is $r$ with $W=W_r$.
\end{enumerate}
\end{thm}
\begin{proof}[Sketch of the proof]
Let $\{W_r' \mid r \in V\}$ be the definable collection of 
pseudo-grounds of $V$ as in Theorem \ref{2.9}.
If $W_r'$ is a ground of $V$ then put $W_r=W_r'$,
and $W_r=V$ otherwise.
By Theorem \ref{2.6},
the definable collection $\{W_r \mid r \in \}$ comprises all grounds of $V$.
\end{proof}
\begin{note}\label{notedayo}
In addition,  we have the following by Theorem \ref{2.2}:
Suppose $\CLS$. If the formula $\varphi$ defines all grounds as above, 
then, in all generic extensions, 
the formula $\varphi$ defines all its  grounds by the same way.
\end{note}

\begin{lemma}\label{2.10}
Suppose $\CLS$.
Then there is a first order formula $\psi(x,y,z)$ in the language of set-theory
such that for every poset $\bbP$, $\bbP$-name $\dot s$, and a set $t$,
$\Vdash_\bbP$``\,$W_{\dot s}^{V[\dot G]} =\check W_t'$''
if and only if $\psi(\bbP, \dot s,t)$ holds.
Hence 
the statement $\Vdash_\bbP$``\,$\check W_t'$ is a ground of the universe'' is 
a first order assertion as $\exists \bbP \exists \dot s \psi(\bbP, \dot s, t)$.
\end{lemma}
\begin{proof}
Let $\psi(\bbP,\dot s,t)$ be the following sentence:
There are \LS cardinals $\ka<\la$ which are limits of \LS cardinals and $\bbP \in V_\ka$
such that:
\begin{enumerate}
\item  $\Vdash_\bbP$``$W_{\dot s}^{V[\dot G]}$ is a ground of the universe $V[\dot G]$ via 
a poset $\bbQ \in (W_{\dot s}^{V[\dot G]})_\ka$''.
\item For every $p \in \bbP$, there is $q \le p$ and a set $X$
such that $q \Vdash_\bbP$``$X=(W_{\dot s}^{V[\dot G]})_\la$'',
$X=(W_{u}')_\la$ where $u=\seq{X, \la, \ka}$,
and $W_{u}'=W_t'$.
\end{enumerate}
We see that this $\psi$ works.

First, suppose $\Vdash_\bbP$``$W_{\dot s}^{V[\dot G]} =\check W_t'$''.
By $\CLS$,
there are large \LS cardinals $\ka<\la$ which are limits of \LS cardinals
such that  $\Vdash_\bbP$``$W_{\dot s}^{V[\dot G]}$ is a ground of the universe $V[\dot G]$ via 
poset $\bbQ \in (W_{\dot s}^{V[\dot G]})_\ka$''.
We see that $\ka$ and $\la$ witness $\psi(\bbP, \dot s, t)$.

By Theorems \ref{2.2} and \ref{2.6}, the following hold:
\begin{enumerate}
\item $\Vdash_\bbP$``$\ka$ and $\la$ are \LS cardinals''.
\item $\Vdash_\bbP$``$W_{\dot s}^{V[\dot G]}$ has the $\ka$-norm covering and 
the $\ka$-norm approximation properties for $V[\dot G]$''.
\item $\Vdash_\bbP$``$\check V$ has the $\ka$-norm covering and 
the $\ka$-norm approximation properties for $V[\dot G]$''.
\end{enumerate}
Note that whenever $G$ is $(V,\bbP)$-generic and $s=\dot s^G$,
we have $W_t'=W_s^{V[G]} \subseteq V \subseteq V[G]$.
Since $W_t'=W_s^{V[G]}$ has the $\ka$-norm covering and the $\ka$-approximation properties for $V[G]$,
we know that $W_t'=W_s^{V[G]}$
 has the $\ka$-norm covering and the $\ka$-approximation properties for $V$ by
Lemma \ref{2.7}.

Now take $p \in \bbP$.
Since $\Vdash_\bbP$``$W_{\dot s}^{V[\dot G]} =\check W_t' \subseteq \check V$'',
we can choose $q \le p$ and a set $X$ such that
$q \Vdash_\bbP$``$X= (W_{\dot s}^{V[\dot G]})_\la$''.
Put $u=\seq{X, \la, \ka}$. 
$W_t'$ has the $\ka$-norm covering and the $\ka$-norm approximation properties for $V$
and $X=(W_t')_\la$.
Hence $W_t'=W_u'$ by the definition of $W_u'$ and Theorem \ref{2.5}.

For the converse,
suppose $\psi(\bbP, \dot s, t)$ holds.
Fix \LS cardinals $\ka<\la$ witnessing $\psi(\bbP, \dot s, t)$.
Fix $p \in \bbP$, and take $q \le p$ and  a set $X$
such that 
$q \Vdash_x$``$X=(W_{\dot s}^{V[\dot G]})_\la$''.
Let $u=\seq{X, \la,\ka}$.
We know that $W_t'=W_u'$, $(W_u')_\la=X$, and $W_t'=W_u'$ has the $\ka$-norm covering and the $\ka$-norm approximation properties for $V$.

Now take a $(V,\bbP)$-generic $G$.
In $V[G]$, $\ka$ and $\la$ remain \LS cardinals,
and $W_s^{V[G]}$ has the $\ka$-norm covering and the $\ka$-norm approximation properties
for $V[G]$.
Moreover $V$ has the $\ka$-norm covering and the $\ka$-norm approximation properties for $V[G]$.
Then so does $W_u'$ for $V[G]$ by Lemma \ref{2.7}.
Applying Theorem \ref{2.5},
we have $W_t'=W_u'=W_s^{V[G]}$.
\end{proof}

\begin{define}[Blass \cite{B}]
The principle $\SVC$ (Small Violation of Choice)
is the assertion that there is a set $S$ such that for every set $X$, there is an ordinal
$\alpha$ and a surjection from $S \times \alpha$ onto $X$.
\end{define}

\begin{thm}[\cite{B}]
The following are equivalent:
\begin{enumerate}
\item $\SVC$ holds.
\item $\AC$ is forceable, that is,
there is a poset $\bbP$ which forces $\AC$.
\end{enumerate}
\end{thm}
By this theorem, we know that 
$\SVC$ is absolute between $V$, any ground, and any generic extension of $V$.

\begin{thm}[\cite{U2}]
Suppose $\SVC$ holds.
Then $\CLS$ holds, in particular all grounds of $V$ are uniformly definable.
\end{thm}

\section{Symmetric grounds}
In this section, we study a characterization of
symmetric grounds without taking \emph{some} generic extension,
and we prove the uniform definability of symmetric grounds under $\CLS$.

\begin{lemma}\label{3.1}
Let $M, N$ be models of $\ZF$ and suppose $M$ is a ground of $V$,
say $V=M[G]$.
If $M \subseteq N \subseteq V$ and $N$ is a symmetric submodel of $V=M[G]$,
then $N$ is of the form $M(X)$ for some $X$.
\end{lemma}
\begin{proof}
By Theorem \ref{1.6}, 
$N$ is of the form $\HOD_{M(Y)}$ for some $Y$.
Since $V=M[G]=M(Y)(G)$, we know that $N=\HOD_{M(Y)}$ is of the form $M(Y)(Z)$ for some $Z$ with $Z \subseteq M(Y)$ by Theorem \ref{1.3}. Then $N=M(\{Y,Z\})$.
\end{proof}

\begin{lemma}\label{3.2}
Let $M, N$ be models of $\ZF$ and suppose $M$ is a ground of $V$, say $V=M[G]$.
Suppose  $M \subseteq N \subseteq V$, 
$M$ is a class of $N$, and $N$ is of the form $M(X)$ for some $X$.
Then there is a large limit $\alpha$ such that
whenever $G'$ is $(N, \col(N_\alpha))$,
there is an $(M, \col(M_\alpha))$-generic $H$
such that $N[G']=M[H]$ and $N$ is a symmetric submodel of $M[H]$.
\end{lemma}
\begin{note}
As in Note \ref{1.4}, the conclusion of this proposition
can be read as follows:
There is some $\alpha$ such that,
in $N$, $\col(N_\alpha)$ forces ``
there is an $(\check M, \col(M_\alpha))$-generic $H$
such that $\check {M}[H]$ is the  universe,
and $\check N$ is a symmetric submodel of $\check M[H]$''.
\end{note}

\begin{proof}
By Theorem \ref{1.2}, $V=M[G]$ is a generic extension of $N$.
Applying Lemma \ref{1.5+}, there is a large limit $\alpha$ such that
if $G$ is $(V, \col(V_\alpha))$-generic,
there is an $(M, \col(M_\alpha))$-generic $H$ and
an $(N, \col(N_\alpha))$-generic $H'$ 
with $V[G]=M[H]=N[H']$.
Then, in $N$, $\col(N_\alpha)$ forces that
``there is an $(\check M, \col(M_\alpha))$-generic $H$ such
that $\check M[H]$ is the universe''.
Hence whenever $G'$ is $(N, \col(N_\alpha))$-generic, 
there is an $(M, \col(M_\alpha))$-generic $H$
such that $N[G']=M[H]$.
We shall see that $N$ is of the form $(\HOD_{M(X)})^{M[H]}$, hence $N$ is a symmetric submodel of $M[H]$ by Theorem \ref{1.6}.
Since $N=M(X)$, it is clear that $N \subseteq (\HOD_{M(X)})^{M[H]}$.
Because $\col(N_\alpha)$ is a weakly homogeneous ordinal definable poset in $N$,
we have $(\HOD_{M(X)})^{M[H]}=(\HOD_N)^{N[G']} \subseteq N$. 
\end{proof}


We have the following observation.
Roughly speaking, it asserts that $W$ is a symmetric ground of $V$ if and only if 
$W$ is a ground of \emph{some} generic extension of $V$.

\begin{prop}\label{3.4}
Let $M\subseteq N$ be  models of $\ZF$, and suppose $M$ is a class of $N$.
\begin{enumerate}
\item Suppose $M$ is a ground of $V$, say $V=M[G]$.
If $N$ is a symmetric submodel of $M[G]=V$,
then there are a poset $\bbP\in N$ and $\bbQ \in M$
such that, in $N$, $\bbP$ forces ``there is an $(\check M, \bbQ)$-generic $H$
such that $\check M[H]$ is the  universe''.
\item If there are a poset $\bbP \in N$ and $\bbQ \in M$
such that, in $N$, $\bbP$ forces ``there is an $(\check M, \bbQ)$-generic $H$
such that $\check M[H]$ is the  universe'',
then $N$ is a symmetric submodel of some generic extension of $M$.
\end{enumerate}
\end{prop}
\begin{proof}
(1). By Lemma \ref{3.1},
$N$ is of the form $M(X)$ for some set $X \in N$.
Then by Lemma \ref{3.2},
there is some large $\alpha$
such that, in $N$, $\col(N_\alpha)$ forces that
``there is an $(\check M, \col(\check M_\alpha))$-generic $H$
such that $\check M[H]$ is the  universe''.

(2). Take an $(N, \bbP)$-generic $G$.
In $N[G]$, there is $H \subseteq \bbQ \in M$ with $N[G]=M[H]$.
$M \subseteq N \subseteq N[G]=M[H]$,
hence $N$ is of the form $M(X)$ for some $X \in M[H]$
by Theorem \ref{1.2}, and $N$ is a symmetric submodel of some generic extension of $M[H]$ by Lemma \ref{3.2}.
\end{proof}
Using this proposition, we can obtain a formal definition of symmetric grounds.
\begin{define}\label{3.5}
Let $W$ be a model of $\ZF$. Let us say that $W$ is a \emph{symmetric ground} of $V$,
or $V$ is a \emph{symmetric extension} of $W$,
if there are posets $\bbP \in V$ and $\bbQ \in W$ such that
$\bbP$ forces that
``there is a $(\check W, \bbQ)$-generic $H$
such that $\check W[H]$ is the  universe''.
\end{define}

By Proposition \ref{3.4},
our notion of symmetric grounds and extensions coincide with the standard definition of symmetric submodels
and extensions.
Moreover our notion of symmetric extension is equivalent to
quasi-generic extension in Grigorieff \cite{G}.

\begin{note}
In Definition \ref{3.5}, 
$V$ need not be a symmetric extension of $W$ via  a poset $\bbQ$.
\end{note}
\begin{note}
If $W$ is a symmetric ground of $V$,
then $V=M(X)$ for some set $X$ by Lemma \ref{3.1}.
But the converse does not hold;
For instance, if $0^\#$ exists,
then $V=L(0^\#)=L[0^\#]$ is not a symmetric extension  of $L$.
\end{note}

By Lemma \ref{1.5+}, posets $\bbP$ and $\bbQ$ can be
Levy collapsings.
\begin{lemma}\label{3.7+}
Let $W$ be a model of $\ZF$.
Then $W$ is a symmetric ground  of $V$ if and only if
for every sufficiently large limit ordinal $\alpha$,
$\col(V_\alpha)$ forces ``there is a $(\check W, \col(\check W_\alpha))$-generic $H$
such that $\check{W}[H]$ is the  universe''.
\end{lemma}

Now we  obtain a first order definition of symmetric grounds under $\CLS$.
\begin{thm}\label{3.7}
Suppose $\CLS$ (e.g., $\SVC$ holds).
Then there is a first order formula $\varphi(x,y)$ of set-theory such that:
\begin{enumerate}
\item For every set $r$, $\overline{W}_r=\{x \mid \varphi(x,r)\}$ is a symmetric ground of $V$ with $r \in \overline{W}_r$.
\item For every symmetric ground $W$ of $V$, there is $r$ with
$W=\overline{W}_r$.
\end{enumerate}
\end{thm}
\begin{proof}
Let $\{W_r'\mid r \in V \}$ be the collection of all pseudo-grounds defined as in Theorem \ref{2.9}.
Define $\overline{W}_r$ as follows:
If there is a poset $\bbP$ such that
$\Vdash_\bbP$``$\check W_{r}'$ is a ground of the universe'',
then $\overline{W}_r=W_r'$. Otherwise $\overline{W}_r=V$.
By Lemma \ref{2.10},
 $\{\overline{W}_r \mid r \in V\}$ is 
a first order definable collection.
We check that $\{\overline{W}_r \mid r \in V\}$ is the collection of all symmetric grounds of $V$.

For $r \in V$, if $\overline{W}_r \neq V$ then $\overline{W}_r=W_r'$ and there is a poset $\bbP$
such that $\Vdash_\bbP$``$\check W_r'$ is a ground of the universe''.
Hence $W_r'$ is a ground of a generic extension of $V$ via $\bbP$,
so $\overline{W}_r=W_r'$ is a symmetric ground of $V$.
For the converse,
suppose $W$ is a symmetric ground of $V$.
We can choose a generic extension $V[G]$ of $V$ via poset $\bbP$
such that $W$ is a ground of $V[G]$.
We can take a large \LS cardinal $\ka$
such that $\ka$ is \LS in V[G],
and $W$ have the $\ka$-norm covering and the $\ka$-norm approximation properties for $V[G]$.
Since $W \subseteq V \subseteq V[G]$,
$W$ has the $\ka$-norm covering and the $\ka$-norm approximation properties for $V$.
Hence $W=W_r'$ for some $r \in V$, and in $V$ we can choose a poset $\bbQ$
with $\Vdash_\bbQ$``$\check W_r'$ is a ground of the universe'',
so $W=\overline{W}_r$.
\end{proof}

Hence in $\ZFC$, we can always define all symmetric grounds uniformly.

\section{Some properties of symmetric extensions and grounds}

In this section, we make some observations and prove some useful properties of
symmetric extensions and grounds.


The following is partially proved by Grigorieff (\cite{G}).

\begin{lemma}\label{4.3}
Let $M$ and $N$ be models of $\ZF$ with $M \subseteq N \subseteq V$.
Then any two of the following conditions imply the third:
\begin{enumerate}
\item $N$ is a symmetric extension of $M$.
\item $V$ is a symmetric extension of $N$.
\item $V$ is a symmetric extension of $M$.
\end{enumerate}
\end{lemma}
\begin{proof}
(1)$\land$(2) $\Rightarrow$ (3).
By Lemma \ref{3.7+},
there is some large limit ordinal $\alpha$
such that:
\begin{itemize}
\item In $N$,
$\col(N_\alpha)$ forces that ``there is an $(\check M , \col(\check{M}_\alpha))$-generic $H$
such that $\check M[H]$ is the  universe''.
\item In $V$,
$\col(V_\alpha)$ forces that ``there is an $(\check N, \col(\check{N}_\alpha))$-generic $H$
such that $\check N[H]$ is the  universe''.
\end{itemize}
Take a $(V, \col(V_\alpha))$-generic $G$.
Then there is an $(N,\col(N_\alpha))$-generic $H$
with $V[G]=N[H]$,
and there is an $(M, \col(M_\alpha))$-generic $H'$
with $M[H']=N[H]$.
Hence $M$ is a ground of the generic extension $V[G]$ of $V$,
and $M$ is a symmetric ground of $V$.

(1)$\land$(3) $\Rightarrow$ (2). 
Since $V$ is a symmetric extension of $M$,
there is a generic extension $M[G]$ of $M$ such that
$V \subseteq M[G]$ and $V=M(X)$ for some $X \in V$ by Lemma \ref{3.1}.
By the same argument, $N=M(Y)$ for some $Y \in N$.
We know $M \subseteq N=M(Y) \subseteq V=M(X) \subseteq M[G]$.
Then $M[G]$ is a generic extension of $N=M(Y)$ by
Theorem \ref{1.2}.
We know $V=N(X)$, hence $V$ is a symmetric extension of $N$ by
Lemma \ref{3.2}.

(2)$\land$(3) $\Rightarrow$ (1). By Lemma \ref{3.7+}, 
there is some large $\alpha$ such that 
whenever $G$ is $(V, \col(V_\alpha))$-generic, we have that
$V[G]=N[H]=M[H']$ for some $H \subseteq \col(N_\alpha) \in N$ and $H' \subseteq \col(M_\alpha) \in M$.
Hence $M$ is a ground of $N[H]$, so $M$ is a symmetric ground of $N$.
\end{proof}
By this lemma,
the ``is a symmetric extension of'' relation is transitive.

It is  known that if $M \subseteq N \subseteq  M[G]$ are models of $\ZFC$ and
$M[G]$ is a generic extension of $M$,
then $M$ is a ground of $N$, and $N$ is of $M[G]$ (compare Theorem \ref{1.2}).
We prove a variant of this fact in our context.
\begin{lemma}\label{5.3}
Suppose $M$ is a symmetric ground of $V$.
Then for every model $N$ of $\ZF$ with $M \subseteq N \subseteq V$,
the following are equivalent:
\begin{enumerate}
\item $N$ is a symmetric ground of $V$.
\item $N$ is of the form $M(X)$ for some set $X \in V$.
\item $N$ is a symmetric extension of $M$.
\end{enumerate}
\end{lemma}
\begin{proof}
(1) $\iff$ (3) is immediate from Lemma \ref{4.3}.

(2) $\Rightarrow$ (1). Since $M$ is a symmetric ground,
there is a generic extension $V[G]$ of $V$ such that
$M$ is a ground of $V[G]$.
Then $N=M(X)$ is a ground of $V[G]$ by Theorem \ref{1.2},
so $N$ is a symmetric ground of $V$.

(1) $\Rightarrow$ (2) is Lemma \ref{3.1}.
\end{proof}
It is also known that the Bristol model (\cite{K})
is an intermediate model $N$ between the constructible universe $L$ 
and the Cohen forcing extension of $L$ but $N$ does not satisfy (1)--(3) of Lemma \ref{5.3}.

We can characterize $\SVC$ in terms of symmetric grounds.
Note that under $\AC$, all grounds are uniformly definable.
\begin{thm}[\cite{U1}]\label{DDG}
Suppose $V$ satisfies $\AC$,
and let $\{W_r \mid r \in V\}$ be the collection of all $\ZFC$-grounds of $V$.
Then for every set $X$, there is
a set $s$ such  that $W_s \subseteq \bigcap_{r \in X} W_r$.
\end{thm}

\begin{prop}\label{4.5}
The following are equivalent:
\begin{enumerate}
\item $\SVC$ holds.
\item There is a symmetric ground satisfying $\AC$.
\item There is a symmetric ground $W$ satisfying $\AC$ and $V=W(X)$ for some set $X$.
\end{enumerate}
\end{prop}
\begin{proof}
(1) $\Rightarrow$ (2). Take a poset $\bbP$ which forces $\AC$.
Take a $(V,\bbP \times \bbP)$-generic $G \times H$.
Both $V[G]$ and $V[H]$ are grounds of $V[G][H]$ and satisfy $\AC$.
By Theorem \ref{DDG}, there is a ground $W$ of $V[G][H]$ such that
$W$ satisfies $\AC$ and $W \subseteq V[G] \cap V[H]$.
Since $V[G] \cap V[H]=V$ (e.g., see \cite{FHR}), we have $W \subseteq V \subseteq V[G][H]$.
Then $V$ is of the form $W(X)$ for some set $X$ by Theorem \ref{1.2}.
Now $V$ is a symmetric extension of $W$ by Lemma \ref{5.3}.

(2) $\Rightarrow$ (1).
There is a generic extension $W[G]$ of $W$ such that
$V$ is a ground of $W[G]$.
$W[G]$ satisfies $\AC$, so $W[G]$ witnesses that $\SVC$ holds in $V$.

(2) $\Rightarrow$ (3).
If $W$ is a symmetric ground satisfying $\AC$,
then $W$ is a ground of some generic extension $V[G]$ of $V$.
Applying Theorem \ref{1.2},
we have that $V=W(X)$ for some set $X \in V$.

(3) $\Rightarrow$ (2) is trivial.
\end{proof}

Next we prove the absoluteness of
$\CLS$ and $\SVC$ between all symmetric grounds and symmetric extensions.

\begin{prop}\label{4.7}
Let $W$ be a symmetric ground of $V$.
Then $\CLS$ and $\SVC$ are absolute between $W$ and $V$.
\end{prop}
\begin{proof}
Since $V$ is a symmetric extension of $W$,
we have that $V$ is a ground of some  generic extension $W[G]$ of $W$.
Since $\CLS$ and $\SVC$ are absolute between all grounds and generic extensions,
it is also absolute between $W$, $V$, and $W[G]$.
\end{proof}

By the absoluteness of $\CLS$ and Theorem \ref{3.7},
we have: 
\begin{cor}
Suppose $\CLS$ holds.
Then $V$ is definable in its symmetric extensions with parameters from $V$.\footnote{A. Karagila independently obtained this result assuming that $V$ satisfies $\AC$.}
\end{cor}
In particular, under $\AC$, $V$ is always definable in its symmetric extensions.

\begin{note}
As the uniform definability of grounds, 
under $\CLS$, if the formula $\varphi$ defines all symmetric grounds as in Theorem \ref{3.7},
then in all symmetric grounds and symmetric extensions of $V$,
$\varphi$ also defines all its symmetric grounds.
This follows from the absoluteness of $\CLS$ and the proof of Theorem \ref{3.7}.
\end{note}

Finally, we observe the structure of the symmetric models under $\AC$.

\begin{lemma}\label{5.8}
Suppose $V$ satisfies $\AC$. Then every symmetric ground satisfying $\AC$ is a ground.
\end{lemma}
\begin{proof}
Let $W$ be a symmetric ground of $V$.
Then $W$ is a ground of some generic extension $V[G]$ of $V$.
$V$, $V[G]$, and $W$ are models of $\ZFC$ with $W \subseteq V \subseteq V[G]$.
Then $W$ is a ground of $V$.
\end{proof}

\begin{cor}
Suppose $V$ satisfies $\AC$.
Then for every model $M$ of $\ZF$,
the following are equivalent:
\begin{enumerate}
\item $M$ is a symmetric ground of $V$.
\item $M$ is a ground of $V$.
\item There is a $\ZFC$-ground $W$ of $V$
and a set $X$ such that $M=W(X)$.
\end{enumerate}
\end{cor}
\begin{proof}
(3) $\Rightarrow$ (2). If $M=W(X)$ for some ground $W$ of $V$,
then $M$ is a  ground by Theorem \ref{1.2}.

(2) $\Rightarrow$ (1) is trivial.

(1) $\Rightarrow$ (3). Suppose $M$ is a symmetric ground of $V$.
Since $V$ satisfies $\SVC$ trivially,
$\SVC$ holds in $M$ as well by Proposition \ref{4.7}.
Then $M$ has a symmetric ground $W$ satisfying $\AC$ and 
a set $X$ with $M=W(X)$ by Proposition \ref{4.5}.
By Lemma \ref{4.3}, $W$ is a symmetric ground of $V$,
and in fact it is a ground by Lemma \ref{5.8}.
\end{proof}

We also note the following, which contrasts with Lemma \ref{5.3}.
\begin{cor}
Let $M$ be a symmetric ground of $V$.
Then for every model $N$ of $\ZF+\SVC$ with $M \subseteq N \subseteq V$,
$N$ is a symmetric ground of $V$, and $N$ is a symmetric extension of $M$.
\end{cor}
\begin{proof}
By Proposition \ref{4.5}, 
$N$ has a symmetric ground $W$ satisfying $\AC$ and $N$ is of the form $W(X)$.
Since $M$ is a symmetric ground of $V$,
there is a set $Y$ with $V=M(Y)$.
Then clearly $V=W(\{X,Y\})$, and $\SVC$ holds in $V$. Force with $\col(\mathrm{trcl}(\{X,Y\}))$ over $V$.
This produces a generic extension $V[G]$ of $V$ 
such that $V[G]$ which satisfies $\AC$.
Moreover $\SVC$ holds in $M$ as well by Proposition \ref{4.7}.
Thus there is a symmetric ground $U$ of $M$
such that $U$ satisfies $\AC$ by Proposition \ref{4.5} again.
Then $U$ is a symmetric ground of $V[G]$, hence is a ground of $V[G]$ by Lemma \ref{5.8}.
Now $U \subseteq W \subseteq V[G]$ and $V[G]$ is a generic extension of $U$,
so $W$ is a ground of $V[G]$.
Therefore $N=W(X)$ is a symmetric ground of $V[G]$, and 
so $N$ is a symmetric ground of $V$.
In addition $N$ is a symmetric extension of $M$ by Lemma \ref{5.3}.
\end{proof}

%

\section{The downward directedness of symmetric grounds}
In $\ZFC$,
the collection of $\ZFC$-grounds  is downward directed (Theorem \ref{DDG}).
Unlike in the $\ZFC$-context, under $\ZF$+$\SVC$, 
it is possible that $V$ has two grounds which have no common ground
(see \cite{U2}).
On the other hand, we prove that under $\SVC$,  
all symmetric grounds are downward directed.

\begin{prop}\label{6.1}
Suppose $\SVC$ holds (hence all symmetric grounds are uniformly definable),
and let $\{\overline{W}_r \mid r \in V\}$ be the collection of all symmetric grounds.
Then for every set $X$,
there is a symmetric ground $W$ of $V$
such that $W \subseteq \bigcap_{r \in X} \overline{W}_r$,
and $W$ is a symmetric ground of each $\overline{W}_r$ for $r \in X$.
\end{prop}
\begin{proof}
By Corollary \ref{4.7}, 
$\SVC$ holds in $\overline{W}_r$ for every $r \in X$.
By Proposition \ref{4.5},
each $\overline{W}_r$ has a symmetric ground satisfying $\AC$.
Every symmetric ground of $\overline{W}_r$ is a symmetric ground of $V$ as well
by Lemma \ref{4.3}.
Hence we can find a set $Y$ such that 
for each $s \in Y$, $\overline{W}_s$ satisfies $\AC$,
and for every $r \in X$ there is $s \in Y$
with $\overline{W}_s \subseteq \overline{W}_r$.

Take a generic extension $V[G]$ of $V$ such that $\AC$ holds in $V[G]$.
By Lemma \ref{5.8}, for each $s \in Y$, $\overline{W}_s$ is a ground of $V[G]$.
By Theorem \ref{DDG},
there is a ground $W$ of $V[G]$ 
such that $W$ satisfies $\AC$ and $W \subseteq \overline{W}_{s}$ for every $s \in Y$.
Now, for each $s \in Y$,
we have $W \subseteq \overline{W}_s \subseteq V \subseteq V[G]$.
Since $V[G]$ is a generic extension of $V$ and $W$,
we know that $W$ is a symmetric ground of $V$.
By the choice of $Y$, we have that $W \subseteq \bigcap_{r \in X} \overline{W}_r$.
In addition, by Lemma \ref{4.3},
$W$ is a symmetric ground of $\overline{W}_r$ for every $r \in X$.
\end{proof}

\section{Mantles}
If all grounds are uniformly definable as in Theorem \ref{unif. def.  grounds},
then we can define the intersection of all grounds.
We say that the intersection of all grounds is the \emph{$\ZF$-mantle}, or simply the \emph{mantle}, 
and the intersection of all $\ZFC$-grounds is the \emph{$\ZFC$-mantle}.
The mantle and $\ZFC$-mantle are  important objects in set-theoretic geology.

Suppose $\CLS$, and let $\varphi$ be the formula 
which defines all grounds uniformly as in Theorem \ref{unif. def.  grounds}.
As mentioned in Note \ref{notedayo}, in all generic extensions of $V$, 
$\varphi$ also defines all its grounds by the same way.
Then we can define the \emph{generic mantle},
which is the intersection of all grounds of all generic extensions.

\begin{define}[\cite{FHR}, \cite{U2}]
Suppose $\CLS$.
\begin{enumerate}
\item The \emph{$\ZF$-mantle}, or simply the \emph{mantle}, $\bbM^{\ZF}$ is the intersection of all grounds,
and the \emph{$\ZFC$-mantle} $\bbM^{\ZFC}$ is the intersection of all $\ZFC$-grounds.
\item The \emph{$\ZF$-generic mantle}, or the \emph{generic mantle}, $g\bbM^{\ZF}$ is the intersection of all grounds of
all generic extensions,
and the \emph{$\ZFC$-generic mantle} $g\bbM^{\ZFC}$
is the intersection of all $\ZFC$-grounds of
all generic extensions.
\end{enumerate}
\end{define}
It is known that the ($\ZFC$-)mantle and the ($\ZFC$-)generic mantle are  parameter free definable 
transitive classes (under $\CLS$).
In $\ZFC$, it is known that the $\ZFC$-mantle is a model of $\ZFC$,
and it coincides with the $\ZFC$-generic mantle (\cite{FHR}, \cite{U1}).
In $\ZF(+\CLS)$, the generic mantle is a model of $\ZF$, and it is a forcing invariant model
(\cite{FHR}, \cite{U2}).
\begin{thm}[\cite{U2}]
If $\SVC$ holds,
then the generic mantle $g\bbM^{\ZF}$ is a model of $\ZFC$.
Moreover, $g\bbM^{\ZF}$ coincides with $g\bbM^{\ZFC}$.
\end{thm}

Let us consider the intersection of all symmetric grounds.
\begin{define}
Suppose $\CLS$.
The \emph{symmetric mantle} $s\bbM$ is the intersection of all symmetric grounds.
\end{define}
As the mantle and the generic mantle,
the symmetric mantle is a parameter-free first order definable transitive class
containing all ordinals.
By the definitions, we have $g \bbM^{\ZF} \subseteq s\bbM \subseteq \bbM^{\ZF}$.

%
%
%
%
%

\begin{prop}
If $\SVC$ holds,
then the symmetric mantle coincides with the generic mantle.
In particular the symmetric mantle is a model of $\ZFC$.
\end{prop}
\begin{proof}
Take  a set $x$ and suppose $x \notin g\bbM^{\ZF}$.
Then there is a generic extension $V[G]$ of $V$ and a ground $W$ of $V[G]$ such that
$x \notin W$. Note that  $W$ need not  be a symmetric ground of $V$,
but $W$ and $V$ are symmetric grounds of $V[G]$.
By Proposition \ref{6.1}, there is $M$ which is a common symmetric ground of $W$ and $V$.
Then $x \notin M$, hence $x \notin s\bbM$.
\end{proof}
However, beside the generic mantle, it is not known if $s \bbM$ and $\bbM^{\ZF}$ are always models of $\ZF$.

\section{Questions}
To conclude this paper let us pose some questions.
\begin{question}
Without any assumption, are all symmetric grounds uniformly definable?
\end{question}
This question is almost equivalent to the uniform definability of all grounds in $\ZF$,
this is also asked in \cite{U2}.

\begin{question}
Is the symmetric mantle a model of $\ZF$?
\end{question}

\begin{question}
Does the symmetric mantle always coincide with the generic mantle?
\end{question}

As in the context of  $\ZFC$,
the downward directedness of symmetric grounds
yields a positive answer of this question.

\begin{question}
Without any assumption, are all symmetric grounds downward directed?
\end{question}

\begin{question}
Is it consistent that the symmetric mantle is strictly smaller than the mantle?
\end{question}

Under $V=L$, there is no proper symmetric ground.
However we do not know a choiceless model which has no proper symmetric ground.

\begin{question}
Is it consistent that $V$ does not satisfy $\AC$, and has no proper symmetric ground?
\end{question}
If there exists such a model, then $\SVC$ must fail in the model by Proposition \ref{4.5}.
A similar question is:
\begin{question}
Is it consistent that $V$ has a proper symmetric ground but has no
proper ground?
\end{question}
If such a model exists, then $\AC$ must fail in the model, and 
the symmetric mantle of the model is strictly smaller than
the mantle.

\section*{Acknowledgments}
The author would like to thank Asaf Karagila for many valuable comments.
The author also thank the referee for many useful comments and suggestions.
This research was supported by
JSPS KAKENHI Grant Nos. 18K03403 and 18K03404.
%
%
%
%
%
%
%
%
%
%
%
%
%
%

\end{document}